\documentclass[11pt]{article}

\usepackage[pagebackref,colorlinks=true,pdfpagemode=none,urlcolor=blue,
linkcolor=blue,citecolor=blue]{hyperref}

\usepackage{amsmath,amsfonts,amssymb,amsthm,color}
\usepackage{mathrsfs}
\usepackage{enumitem}
\usepackage{color}

\newtheorem{theorem}{Theorem}[section]
\newtheorem{lemma}[theorem]{Lemma}
\newtheorem{definition}[theorem]{Definition}

\newtheorem{proposition}[theorem]{Proposition}

\theoremstyle{remark}
\newtheorem{remark}[theorem]{Remark}
\renewenvironment{proof}[1][Proof]{\noindent{\itshape {#1.} } }{$\Box$
\medskip}

\numberwithin{equation}{section}
\newcommand{\R}{\mathbb{R}}

\newcommand{\E}{\mathbb{E}}

\newcommand{\eps}{\varepsilon}
\newcommand{\X}{X}
\newcommand{\D}{\mathcal{D}}
\newcommand{\diam}{\mathrm{rad}}
\newcommand{\absp}{q}
\newcommand{\Em}{E^{\mathrm{m}}}
\newcommand{\Es}{E^{\mathrm{s}}}
\newcommand{\alphas}{\alpha^{\mathrm{s}}}
\newcommand{\As}{A^{\mathrm{s}}}

\newcommand{\h}{\hspace*{.24in}}

\newcommand{\io}{\int_\X}

\title{Quantitative thermo-acoustic imaging: An exact reconstruction formula\thanks{\footnotesize This work
was supported  by ERC Advanced Grant Project MULTIMOD--267184.}}
\author{Habib Ammari\thanks{D\'epartement de Math\'ematiques et Applications, Ecole Normale
Sup\'erieure, 45 Rue d'Ulm, 75005 Paris, France
(habib.ammari@ens.fr, wjing@dma.ens.fr, lnguyen@dma.ens.fr).} \and
Josselin Garnier\thanks{Laboratoire de Probabilit\'es et Mod\`eles
Al\'eatoires \& Laboratoire Jacques-Louis Lions, Universit\'e
Paris VII, 2 Place Jussieu, 75251 Paris Cedex 5, France
(garnier@math.jussieu.fr).} \and Wenjia Jing\footnotemark[2] \and
Loc Hoang Nguyen  \footnotemark[2]}

\begin{document}
\maketitle
\begin{abstract}
This paper aims to  mathematically advance the field of
quantitative thermo-acoustic imaging. Given several
electromagnetic  data sets, we  establish for the first time an
analytical
 formula for reconstructing the absorption coefficient from thermal energy measurements.
 Since the formula involves derivatives of
 the given data up to the third order, it is unstable in the sense that small measurement noises
 may cause large errors. However, in the presence of measurement noise,
 the obtained formula, together with a noise regularization technique, provides
 a good
 initial guess for the true absorption coefficient. We finally correct
 the errors by deriving a reconstruction formula based on the least
  square solution of an optimal control problem and prove that this optimization
  step reduces the errors occurring and enhances the resolution.
\end{abstract}

\bigskip

\noindent {\footnotesize Mathematics Subject Classification
(MSC2000): 35R30, 35R60}


\section{Introduction}
\label{sec:intro}

Hybrid imaging modalities are based on a multi-wave concept.
Different physical types of waves are combined into one
tomographic process to alleviate deficiencies of each separate
type of waves, while combining their strengths. Multi-wave systems
are capable of high-resolution and high-contrast imaging
\cite{AMMARI-08, fink}. Quantitative thermo-acoustic tomography is
an emerging hybrid modality \cite{elson, BalRenUhmannZhou}. It
allows to determine the absorption distribution of a tissue from
boundary measurements of the pressure induced by electromagnetic
heating. Other examples of hybrid modalities are acousto-electric
tomography \cite{AmmariLaurentetal:preprint2012,
AMMARI-BONNETIER-CAPDEBOSCQ-08, Ammarietal:SIAM2011, wenjia,
Capd09, Kuch10, Wang04, otmar2}, magnetic resonance electrical
impedance tomography \cite{SEO02, SEO-WOO-11,
NACHMANN-TAMASAN-TIMONIV-10}, magnetic resonance elastography
\cite{AMMARI-GARAPON-08, muth, seo2}, impedance-acoustic
tomography \cite{GEBAUER-SCHERZER-08}, photo-acoustic \cite{wang,
kuchment_ejam, AMMARI-BOSSY-10}, quantitative photo-acoustic
tomography \cite{AMMARI-BOSSY-11, BalUhmann:ip2010, zhao},
magneto-acoustic imaging \cite{AMMARI-CAPDEBOSCQ-09}, and
vibro-acoustography \cite{FATEMI}.

The aims of this paper are to derive an exact formula for the
absorption coefficient from noiseless thermo-acoustic measurements
and to correct the errors of in the presence of measurement noise.
The former task is motivated by the knowledge of the ratio between
two modified data. For the latter purpose, we show how to
regularize the exact formula and propose an optimal control
algorithm to achieve a resolved image starting from the
regularized one. As far as we know, our exact formula in this
paper together with the one successfully derived in
\cite{Ammarietal:SIAM2011} are among a few exact formulas in
hybrid imaging. Moreover, the fine analysis of the effect of
measurement noise on the image quality and the proof that an
optimal control approach starting from the regularized images
yields a resolved one have never been done elsewhere.

To describe our approach, we employ several notations. Let $\X$ be
a smooth bounded domain in $\mathbb{R}^d,$ $d = 2$ or $3$. Let
$\partial \X$ denote the boundary of $\X$ and let $\nu$ be the
outward normal at $\partial \X$. For $m$ a non-negative integer,
we define the space $H^{m}(\X)$ as the family of all $m$ times
weakly differentiable functions in $L^2(\X)$, whose weak
derivatives of orders up to $m$ are functions in $L^2(\X)$. We let
$H_0^{m}(\X)$ be the closure of $\mathcal{C}^\infty_c(\X)$ in
$H^{m}(\X)$, where $\mathcal{C}^\infty_c(\X)$ is the set of all
infinitely differentiable functions with compact supports in $\X$.
Finally, we introduce the space $H^{1/2}(\partial X)$  of traces
on $\partial X$ of all functions in $H^1(\X)$.

Let $q$ be a positive real-valued function on $X$. Consider the
Helmholtz problem:
\begin{equation}
 \label{eq:helmholtz}
\begin{aligned}
 &(\Delta + k^2 + ik \absp) u = 0, & & x \in \X, \\
 &\nu \cdot \nabla u - ik u = g, & & x \in \partial \X,
\end{aligned}
\end{equation}
which is the scalar approximation of Maxwell's equations. Here,
$k>0$ is the wave number, $g$ is a boundary datum, and $u$ is the
electrical field. The Robin boundary condition approximates
Sommerfeld's radiation condition at high frequencies \cite{EM,KG}.
For simplicity, instead of considering the Helmholtz equation on
the whole Euclidean space with Sommerfeld's radiation condition we
focus on the Helmholtz problem with Robin boundary condition on
the bounded open set $\X$.  Problem (\ref{eq:helmholtz}) is
well-posed in $H^1(\X)$ for all $g \in L^{2}(\partial \X)$. In
fact, writing a variational formulation of (\ref{eq:helmholtz})
shows the uniqueness of a solution to (\ref{eq:helmholtz}), while
the existence of a solution follows from Fredholm's alternative.


The thermo-acoustic imaging problem can be formulated as the
inverse problem of reconstructing the absorption coefficient $q$
from thermo-acoustic measurements $q |u|^2$ in $X$. The quantity
$q |u|^2$ in $X$ is the heat energy due to the absorption
distribution $q$. It generates  an acoustic wave propagating
inside the medium. Finding the initial data in the acoustic wave
from boundary measurements yields the heat energy distribution.
Our aim in this paper is to separate $q$ from $u$. We provide an
explicit formula for reconstructing $q$ from the heat energy $q
|u|^2$ in $X$. As far as we know, our formula is new. Indeed, it
is promising since it can be used as an initial guess to achieve a
resolved image of the absorption distribution in a robust way.

Our first task is to enrich the set of data. Suppose that we have
measurements $q(x) |u_j|^2$ corresponding to linear combinations
of boundary data $g_j$, for $j=1, \ldots, d+1$. We show that one
can construct the set of quantities:
\begin{equation}
\label{eq:dataset}
    \mathcal{E} = \{E_{j}(x) = q(x) u_j(x) \overline{u_1(x)},\,  x \in X  ~|~  j = 1, \ldots, d + 1\},
\end{equation}
where $u_j$ denotes the solution of
    \begin{equation}
        \begin{aligned}
             &(\Delta + k^2 + ik \absp) u_j = 0, & & x \in \X, \\
             &\nu \cdot \nabla u_j - ik u_j = g_j, & & x \in \partial
             \X,
        \end{aligned} \h
    \label{1.2}
    \end{equation}
provided that $(g_j)_{j = 1}^{d + 1}$ is a proper set of
measurements (see Definition \ref{def: proper set of
measurement}). The construction of $E_1$ was completely described
in \cite{BalRenUhmannZhou} and that of $E_j$, $j =2, \ldots, d +
1,$ will be done using Proposition \ref{proposition 2.5}. Noting
that
\[\frac{u_j}{u_1} = \frac{E_j}{E_1}, \h j = 2, \ldots, d + 1,\] we
are able to establish an exact formula for $q$ provided that
$\mathcal{E} = (E_j)_{ j=1}^{d + 1}$ is "good" enough as in
Theorem \ref{Theorem 2.3}. This procedure will be described in
Section \ref{An Initial guess}.

As said, the collected data $\mathcal{E}$ are often corrupted by
measurement noise that varies on very small length scale. This
renders the aforementioned exact formula, which requires
differentiating the data up to third order, completely
unpractical. To solve this issue, we smooth the noise by averaging
the data over a small window and apply the smoothed data to the
exact formula. The resulting function is then shown to be close to
the real one, provided that the width of the averaging window is
properly chosen. We thus view this function as an initial guess
and then perform a further step of least square optimization. The
resulting reconstruction improves the initial guess in the $L^2$
sense.

The rest of the paper is organized as follows. In Section
\ref{sect2} we introduce the notion of a proper set of
measurements and its role to get data $\mathcal{E}$ and some
useful estimates as well. The aim of Section \ref{An Initial
guess} is to provide an explicit formula for reconstructing $q$
when a proper set of measurements is given. In Section
\ref{sec:forw1} we study the Fr\'echet differentiability of the
data with respect to variations of $q$ and prove that the
differential operator is invertible for small enough variations.
In Section \ref{sec:noise} we consider a noise model for the data
and show how to regularize the exact inversion formula in order to
obtain a good initial guess. We also perform a refinement of the
initial guess using an optimal control approach and show that this
procedure yields a resolution enhancement.

\section{Preliminaries} \label{sect2}

Motivated by \cite{Ammarietal:SIAM2011}, we introduce the following concept.
\begin{definition}
    The set $(g_j)_{j=1}^{d + 1} \subset L^2(\partial \X)$ is a \textbf{proper set
    of measurements} of (\ref{eq:helmholtz}) if and only if:
\begin{itemize}
    \item[(i)] $|u_1| > 0$ in $\X$.
    \item[(ii)] The matrix $[u_j, \nabla^T u_j]_{1 \leq j \leq d + 1}$ is invertible for all $x \in
    \X.$
\end{itemize}
    Here, $T$ denotes the transpose and $u_j$ is the solution of (\ref{1.2}).
\label{def: proper set of measurement}
\end{definition}

The following proposition is a direct consequence of Lemma 4.1 in
\cite{BalRenUhmannZhou} and Proposition 3.1 in
\cite{BalUhmann:ip2010}. It plays an important role to prove that
it is possible to find a proper set of measurements.
\begin{proposition} Let $\delta >0$ and $m>d/2$. There exists a positive constant $C$ such that
for any $\xi \in \mathbb{C}^d, \xi \cdot \xi = 0,$ and $|\xi|
> \delta$, and for any $q \in H^m(\X)$, the solution $w$ of
    \begin{equation}
        \Delta w + \xi \cdot \nabla w = -(k^2 + ik q \chi(\X))(1 + w) \quad
        \mbox{\rm in } \mathbb{R}^d,
    \label{eq:opticsolution}
    \end{equation}
where $\chi(\X)$ denotes the characteristic function of $\X$,
satisfies
    \begin{equation}
        \|w\|_{H^m(\X)} \leq \frac{C\|q\|_{H^m(X)}}{|\xi|}.
    \label{ine:boundedness of optic solution}
    \end{equation}

\label{BalRenUhmannZhou}
\end{proposition}

\begin{proposition}
    If $q \in H^{m}(\X)$, $m > 1 + \frac{d}{2}$, then (\ref{eq:helmholtz}) has a proper set of measurements.
\label{proposition 2.3}
\end{proposition}
\begin{proof}
Let $\epsilon$ be a small number. By choosing $\xi$
 such that $\xi \cdot \xi = 0$ and $|\xi|$ is large enough,
we find from the Sobolev embedding theorem and
(\ref{ine:boundedness of optic solution}) that the solution $w$ of
(\ref{eq:opticsolution}) satisfies
\begin{equation}
    \|w\|_{L^{\infty}(\X)} + \|\nabla w\|_{L^{\infty}(\X)} <
    \epsilon.
\label{2.33}
\end{equation}
It is not hard to verify that the function \[u = e^{\xi\cdot x}(1 + w)\] is a solution of
    \begin{equation}
        (\Delta + k^2 + ikq \chi(\X))u = 0
    \label{2.44}
    \end{equation} and it satisfies
    \[
        |u| > |e^{\xi \cdot x}|(1 - \epsilon) > 0.
    \] Choosing $g_1 = \nu \cdot \nabla u - ik u$ on $\partial \X$ gives a solution $u_1$
    of (\ref{1.2}) satisfying part $(i)$ of Definition \ref{def: proper set of measurement}.

Define
\begin{eqnarray*}
    \xi_j &=& n(e_j + ie_{j + 1}), \h j = 1, \ldots ,d - 1,\\
    \xi_d &=& n(e_d + ie_1),
\end{eqnarray*} and
\[
    \xi_{d + 1} = n\bigg(\Big[\sum_{j = 1}^{d - 1}e_j + \sqrt{d - 1}e_d\Big] + i\Big[\sum_{j = 1}^{d - 1}e_j - \sqrt{d - 1}e_d\Big]\bigg),
\]
where $n \gg 1$ and $e_j$ is the $j$th component of the natural
basis of $\mathbb{R}^d.$ Again, it is not hard to verify that
\[
    \xi_j \cdot \xi_j = 0
\] for all $j = 1, \ldots, d + 1,$
and the vectors $(1, \xi_j)_{1 \leq j \leq d + 1}$ are linearly independent in $\mathbb{C}^d$. Hence,
\begin{equation}
    \left|\det\left[
        \begin{array}{cc}
            1 & \xi^T_j
        \end{array}
    \right]_{1 \leq j \leq d + 1} \right| \gg 1,
\label{independence 1}
\end{equation}
provided that $n \gg 1.$
Let $w_j$, $1 \leq j \leq d + 1$, be the solution of
\[
    \Delta w_j + 2\xi_j \cdot \nabla w_j = -(k^2 + ikq \chi(\X))(1 + w_j)
\] and
\[
    u_j = e^{\xi_j \cdot x}(1 + w_j)
\] be the solution of (\ref{2.44}).
We have
\[
    \det\left[ \begin{array}{cc}
        u_j & \nabla^T u_j
    \end{array}
    \right]_{1 \leq j \leq d + 1}
    = e^{\xi_j \cdot x}(1 + w_j)\det \left[ \begin{array}{cc}
        (1 + w_j) & \xi_j^T + \frac{\nabla^T w_j}{1 + w_j}
    \end{array}
    \right]_{1 \leq j \leq d + 1}.
\] Thus, (\ref{2.33}), (\ref{independence 1}), the continuity of the map that sends
a square matrix to its determinant and the choice of large $n$ imply the second part of Definition
\ref{def: proper set of measurement} with
\[
    g_j = \nu \cdot \nabla u_j - ik u_j, \h j = 1, \ldots, d,
\] on $\partial \X$.
\end{proof}

\begin{remark}
    The solution $w$ of (\ref{eq:opticsolution}) is the so-called complex geometric optics solution of (\ref{eq:helmholtz}), which was introduced in \cite{Astala:am2006, SylvesterUhlmann:am1987}. The proof of Proposition \ref{proposition 2.3} was partly motivated by \cite{Triki:ip2010}.
\end{remark}

We next construct the data $\mathcal{E}$, mentioned in Section
\ref{sec:intro}. Let us for the moment accept the following
proposition.
\begin{proposition}
    If $g$ is given, then one can make some measurements to obtain
$\absp(x)|u|^2$, $x \in \X,$ where $u$ solves (\ref{eq:helmholtz}).
\label{proposition 1.1}
\end{proposition}

The following proposition holds.

\begin{proposition}
    Let $g_1, g_2 \in L^2(\partial \X).$ Denote by $u_j$ the solution of
    \begin{equation}
        \begin{aligned}
             &(\Delta + k^2 + ik \absp) u_j = 0, & & x \in \X, \\
             &\nu \cdot \nabla u_j - ik u_j = g_j, & & x \in \partial \X,
        \end{aligned} \h j = 1, 2.
    \end{equation} Then the function
    $
        q(x) u_2(x) \overline{u_1(x)}, x \in \X
    $ can be evaluated.
\label{proposition 2.5}
\end{proposition}
\begin{proof}
    Applying Proposition \ref{proposition 1.1} for $g_1 + g_2$ and then $ig_1 + g_2,$ we obtain the knowledge of
    \[
        q|u_1 + u_2|^2 \mbox{ and } q|iu_1 +  u_2|^2,
    \] respectively. Then the desired data $E_2$ is given by
\begin{equation}
E_2 = \frac{1}{2}(q|u_1 + u_2|^2 - q|u_1|^2 - q|u_2|^2) + \frac{i}{2}(q|i u_1 + u_2|^2 - q|u_1|^2 - q|u_2|^2),
\label{eq:E2def}
\end{equation}
which can be easily verified.
\end{proof}

Let $(g_j)_{j=1}^{d + 1}$ be a proper set of
measurements of (\ref{eq:helmholtz}) and $u_j$ be the solution of
(\ref{eq:helmholtz}) with $g$ replaced by $g_j$. From now on, we
have the knowledge of
\begin{equation}
    \mathcal{E} = (E_{j})_{j=1}^{d + 1},
\label{eq:data}
\end{equation} where $E_{j} = qu_1\overline u_j$, and $\mathcal{E}$ is, therefore, considered as the data to
reconstruct $q$.

We also need the following proposition. It plays an important role
to evaluate the derivative of the data with respect to $q$ in
Section \ref{sec:forw1} as well as some crucial properties.

\begin{proposition}
    Let $q \in L^{\infty}(\X)$ be such that $\inf q >0$.
For all $f \in L^2(\X),$ the problem
    \begin{equation}
        \begin{aligned}
            &(\Delta + k^2 + ik \absp) u = f, & & x \in \X, \\
            &\nu \cdot \nabla u - ik u = 0, & & x \in \partial \X,
        \end{aligned}
    \label{2.88}
    \end{equation} has a unique solution. Moreover, the solution
    satisfies
    \begin{equation}
        \|u\|_{L^2(\X)} \leq \frac{1}{k \inf q} \|f\|_{L^2(\X)}
    \label{2.99}
    \end{equation} and
    \begin{equation}
        \|u\|_{H^1(\X)} \leq \frac{\sqrt{(k^2 + 1) + k\inf q}}{k \inf q} \|f\|_{L^2(\X)}.
    \label{2.1010}
    \end{equation}
\label{proposition 2.6}
\end{proposition}
\begin{proof}
    The well-posedness of (\ref{2.88}) is well-known. Using the test function $u$ in (\ref{2.88}) and
    considering the imaginary and real parts of the resulting equation, we can establish (\ref{2.99})
    and (\ref{2.1010}),
    respectively.
\end{proof}
\section{The exact formula} \label{An Initial guess}

The main aim of this section is to reconstruct $q$ when a proper
set of measurements $(g_j)_{j=1}^{d + 1}$ of (\ref{eq:helmholtz})
and the data $\mathcal{E}$, defined in (\ref{eq:data}), are given.

Let
\begin{equation}
    \alpha_j = \frac{E_{j}}{E_{1}}, \h 2 \leq j \leq d + 1.
\label{2.2}
\end{equation} Then it is not hard to see that
\[
    u_j = \alpha_{j} u_1,
\] for $2 \leq j \leq d + 1.$
We have the following lemma.

\begin{lemma} Let $\beta = \Im (\overline u_1\nabla u_1)$. Then
    \begin{equation}
        -\mbox{div} \, \beta = k E_{1}, \h\mbox{in } \X.
    \label{2.3}
    \end{equation}
\end{lemma}
\begin{proof}
Let $\varphi \in \mathcal{C}_c^{\infty}(\X, \mathbb{R})$ be an
arbitrary function. Then using $\varphi u_1 \in H^1_0(\X)$ as a
test function in
\[
    -\Delta u_1 = (k^2 + ik\absp  )u_1
\] yields
\[
    \io \varphi |\nabla u_1|^2 dx + \io \overline{u_1}\nabla u_1 \cdot \nabla \varphi dx = \io (k^2 + ik\absp ) |u|^2 \varphi dx.
\]
    Taking the imaginary part of the equation above gives
    \[
        -\io \mbox{div} \, (\Im \overline{u_1}\nabla u_1) \varphi dx = \io k \absp|u_1|^2\varphi dx = \io k E_{1}\varphi dx,
    \]
and (\ref{2.3}) follows.
\end{proof}

The following lemma plays an important role in the derivation of
an exact inversion formula for $\absp$.
\begin{lemma}
    For all $2 \leq j \leq d + 1$,
    \begin{equation}
        \nabla \alpha_j \cdot \left(\nabla \log \frac{\absp}{E_{1}} -  \frac{2i\absp\beta}{E_{1}}\right) =  \Delta \alpha_j.
    \label{2.4}
    \end{equation}
\end{lemma}
\begin{proof} Let us fix $j \in \{2, \ldots, d+1 \}$.
Since $u_j$ is a solution of the Helmholtz equation under
consideration,
\begin{eqnarray*}
    (k^2 + ik\absp) \alpha_j u_1 &=& -\Delta \big( \alpha_j u_1 \big)\\
    &=&-\alpha_j \Delta u_1 - u_1\Delta \alpha_j - 2\nabla u_1 \cdot \nabla \alpha_j\\
    &=& (k^2 + ik\absp) \alpha_j u_1 - u_1\Delta \alpha_j  -  2\nabla u_1 \cdot \nabla \alpha_j.
\end{eqnarray*}
Therefore,
\begin{eqnarray*}
    -E_{1}\Delta \alpha_j &=& 2\absp \overline{u_1}\nabla u_1 \cdot \nabla \alpha_j\\
    &=& 2\absp \left(\Re \overline{u_1}\nabla u_1  + i\Im \overline{u_1}\nabla u_1 \right) \cdot \nabla \alpha_j \\
    &=&\absp \left(\nabla |u_1|^2  + 2 i\Im \overline{u_1}\nabla u_1 \right) \cdot \nabla \alpha_j.
\end{eqnarray*}
We have proved that
\[
    -E_{1}\Delta \alpha_j = \absp \left(\nabla |u_1|^2  + 2i\beta \right) \cdot \nabla \alpha_j,\]
 or equivalently,
\begin{equation}
    \absp\nabla|u_1|^2 \cdot \nabla \alpha_j = -E_{1}\Delta \alpha_j - 2i\absp\beta \cdot \nabla \alpha_j.
\label{2.6}
\end{equation}
On the other hand, differentiating the equation $E_{1} = \absp|u_1|^2$ gives
\begin{eqnarray*}
    \nabla E_{1} = \absp\nabla |u_1|^2 + E_{1}\nabla \log \absp.
\end{eqnarray*}
This, together with (\ref{2.6}), implies
\[
    (\nabla E_{1} - E_{1}\nabla \log \absp)\cdot \nabla \alpha_j
     = -E_{1}\Delta \alpha_j - 2i\absp\beta \cdot \nabla \alpha_j,
\]
and (\ref{2.4}), therefore, holds.
\end{proof}

We claim that the set
\[
(\nabla \alpha_j)_{j=2}^{d + 1}
\] is linearly independent for all $x \in \overline \X,$
where $\alpha_j$ was defined in (\ref{2.2}). We only prove this when $d = 2$. The proof when $d$ is larger
than $2$ can be done in the same manner. In fact, the linear independence of $\{\nabla \alpha_2, \nabla \alpha_3\}$
comes from the following
calculation:
\begin{eqnarray*}
    \det\left[
    \begin{array}{c}
                \nabla^T \alpha_2\\
                \nabla^T \alpha_3
    \end{array}
    \right] &=& \frac{1}{u_1^4}\det \left[
    \begin{array}{c}
                u_{1}\nabla^T u_{2} - u_{2}\nabla^T u_{1}\\
                u_{1}\nabla^T u_{3} - u_{3}\nabla^T u_{1}
    \end{array}
    \right]\\
    &=& \frac{1}{u_{1}^4} \left(\det \left[
            \begin{array}{c}
                u_{1}\nabla^T u_{2}\\u_{1}\nabla^T u_{3} - u_{3}\nabla^T u_{1}
            \end{array} \right]  \right. \\ && \h\h\h\h\h\h\h\left.
                - u_{2}\det \left[
            \begin{array}{c}
                \nabla^T u_{1}\\u_{1}\nabla^T u_{3} - u_{3}\nabla^T u_{1}
            \end{array} \right]
        \right)\\
    &=& \frac{1}{u_{1}^3} \left(u_{1}\det \left[
            \begin{array}{c}
                \nabla^T u_{2}\\\nabla^T u_{3}
            \end{array} \right]
                + u_{3}\det \left[
                    \begin{array}{c}
                      \nabla^T u_{1}\\\nabla^T u_{2}
                    \end{array}
                \right]         \right. \\ && \h\h\h\h\h\h\h\left.
                - u_{2}\det \left[
            \begin{array}{c}
                \nabla^T u_{1}\\ \nabla^T u_{3}
            \end{array} \right]
        \right)\\
    &=& \frac{1}{u_{1}^3}\det \left[
        \begin{array}{cc}
            u_{1} & \nabla^T u_{1}\\
            u_{2} & \nabla^T u_{2}\\
            u_{3} & \nabla^T u_{3}
        \end{array}
    \right] \not = 0.
\end{eqnarray*}
Here, part (ii) in Definition \ref{def: proper set of measurement}
has been used. Since the $d\times d$ matrix
\begin{equation}
A = [\nabla^T \alpha_{j+1}]_{1 \leq j \leq d }, \quad \quad A_{jl}= \partial_l \alpha_{j+1} ,
\label{2.888}
\end{equation} is invertible, we can solve system (\ref{2.4}) to get
\begin{equation}
    \nabla \log \frac{\absp}{E_{1}} -  \frac{2i\absp\beta}{E_{1}} =
    a,
\label{2.8}
\end{equation}
where $a$ is the vector $a= A^{-1}[ (\nabla^T {A}^T)^T ]$.

We are now ready to evaluate $\absp$. We first split the real and
the imaginary parts of (\ref{2.8}) to get
\begin{equation}
    \nabla \log \frac{\absp}{E_{1}} = \frac{\nabla \absp}{\absp} -  \nabla \log E_{1} = \Re (a)
\label{2.9}
\end{equation} and
\begin{equation}
    \beta = -\frac{E_{1}\Im (a)}{2 \absp}.
\label{2.10}
\end{equation}
Then, differentiating (\ref{2.10}), we have
\[
    \mbox{div} \, \beta = \frac{E_{1}\Im (a) \cdot \nabla \absp}{2\absp^2} - \frac{\mbox{div} \, (E_{1}\Im (a))}{2\absp}.
\]
This, together with (\ref{2.3}) and (\ref{2.9}), implies
\begin{eqnarray*}
    \absp &=& -\frac{E_{1}(\Re (a) + \nabla \log E_{1}) \cdot \Im (a) - \mbox{div}\, (E_{1}\Im (a))}{2k E_{1}}
        \\&=& - \frac{E_{1}\Re (a) \cdot \Im (a) + \nabla E_{1} \cdot \Im (a)}{2kE_1}
        + \frac{E_{1}\mbox{div} \, \Im (a) + \nabla E_{1} \cdot \Im (a)}{2k E_{1}}\\
        &=& - \frac{\Re (a) \cdot \Im (a) - \mbox{div} \, \Im (a)}{2k}.
\end{eqnarray*}

The results above are summarized in the following theorem.
\begin{theorem}
    Given  a proper set of measurements $(g_j)_{j=1}^{d + 1}$ so that the matrix $A,$ defined in (\ref{2.888}),
     is known and invertible.
 Then,
\begin{equation}
    \absp(x) = \frac{- \Re (a) \cdot \Im (a) +   \mbox{div} \, \Im(a)}{2k},
\label{2.12}
\end{equation}
where $a= A^{-1}[ (\nabla^T {A}^T)^T ]$ and $A= ( \partial_l
\alpha_{j+1} )_{j,l=1,\ldots,d}$.

 \label{Theorem 2.3}
\end{theorem}

\begin{remark}
    Although in the proof of Theorem \ref{Theorem 2.3}, we wrote some
    notations
     requiring the first and second derivatives of $\mathcal{E}$ at a single point
     $x \in \X$, it is not necessary to impose the
    smoothness
    conditions for $\mathcal{E}$. The reason is that one can make the arguments and establish (\ref{2.4}) in the
    weak sense. We argued, using strong forms of differential equations, only for simplicity.
\end{remark}

\begin{remark}
    Formula (\ref{2.12}) is unstable in the sense that if there are some noises occurring when
    we measure the data $E_{j},$ $1 \leq j \leq d + 1$, then $\absp$, given by (\ref{2.12}), might be far away
    from the actual $\absp$ since the right-hand side of (\ref{2.12}) depends on the derivatives of the noise (up to the third order).
\end{remark}

\section{The differentiability of the data map and its inverse}
\label{sec:forw1}

Let $0<q_{\rm min} < q_{\rm max}$. Let
\[
    L^\infty_+(\X)= \bigg\{p \in L^{\infty}(\X):  q_{\rm min} < p< q_{\rm max} \mbox{ in } \X \bigg\}.
\] Then, $L^\infty_+(\X)$ is an open set in $L^\infty(\X)$.
We define the solution and the data map as
\begin{equation}
    \begin{array}{rcl}
    u : L^\infty_+(\X) &\rightarrow& H^1(\X)\\
    q &\mapsto& u[q]
    \end{array}
\label{4.1}
\end{equation}
and
\begin{equation}
    \begin{array}{rcl}
    F : L^\infty_+(\X) &\rightarrow& L^2(\X)\\
    q &\mapsto& F[q] = q|u[q]|^2 ,
    \end{array}
\label{4.2}
\end{equation}
where $u[q]$ is the solution of (\ref{eq:helmholtz}). The map $F$
is well-defined because of the Sobolev embedding theorems and the
fact that $d = 2$ or $3$, which guarantees that $u \in L^4(\X)$.

The main purpose of this section is to study the differential
operator, $\D F[q]$, of $F$ and show that it is invertible
provided that $q_{\max}$ is small enough.

\begin{lemma}
    The map $u$, defined in (\ref{4.1}), is Fr\'{e}chet differentiable in $L^{\infty}_+(X)$. Its derivative at the
    function $q$ is given by
    \begin{equation}
        \D u[q](\rho) = v(\rho), \h \forall \rho \in B_q,
    \end{equation} where $B_q \subset L^{\infty}(\X)$ is an open neighborhood
    of $q$ in $L^{\infty}(\X)$ and $v(\rho)$ is the solution of
    \begin{equation}
        \label{eq:veq}
            \begin{aligned}
                 &(\Delta + k^2 + ik \absp) v = -ik\rho u[\absp], & & x \in \X, \\
                 &\nu \cdot \nabla v - ik v = 0, & & x \in \partial \X.
        \end{aligned}
\end{equation}
Consequently, $F$ is also Fr\'{e}chet differentiable and
    \begin{equation}
        \D F[q] \rho = \rho |u[q]|^2 + 2q \Re (u[q]\overline v(\rho)), \h \forall q \in L^{\infty}_+(\X), \rho \in B_q.
    \label{4.5}
    \end{equation}
\label{lemma 4.1}
\end{lemma}
\begin{proof}
    It is sufficient to show that
        \begin{equation}
            \lim_{\|\rho\|_{L^{\infty}(\X)} \rightarrow 0} h(\rho) =
            0,
        \label{4.6}
        \end{equation}
    where
        \[
            h(\rho) = \frac{\|u[q + \rho] - u[q] - v(\rho)\|_{L^2(\X)}}{\|\rho\|_{L^{\infty}(\X)}}.
        \]
    In fact, since $u[q + \rho] - u[q] - v(\rho)$ solves the problem
        \[
        \begin{aligned}
                 &(\Delta + k^2 + ik \absp)(u[q + \rho] - u[q] - v(\rho)) = -ik\rho (u[q + \rho] - u[\absp]), & & x \in \X, \\
                 &\nu \cdot \nabla (u[q + \rho] - u[q] - v(\rho)) - ik (u[q + \rho] - u[q] - v(\rho)) = 0, & & x \in \partial \X,
        \end{aligned}
        \]
we can apply inequality (\ref{2.99}) to obtain
    \begin{equation}
        \|u[q + \rho] - u[q] - v(\rho)\|_{L^2(\X)} \leq \frac{\|\rho\|_{L^{\infty}(\X)}\|(u[q + \rho] - u[\absp])\|_{L^2(\X)}}{\inf q}.
    \label{4.7}
    \end{equation}
    On the other hand, since $u[q + \rho] - u[q]$ satisfies
        \[
        \begin{aligned}
                 &(\Delta + k^2 + ik (\absp + \rho))(u[q + \rho] - u[q]) = -ik\rho u[q], & & x \in \X, \\
                 &\nu \cdot (\nabla u[q + \rho] - u[q]) - ik (u[q + \rho] - u[q]) = 0, & & x \in \partial \X,
        \end{aligned}
        \]
inequality (\ref{2.99}), again, implies
        \begin{equation}
            \|u[q + \rho] - u[q]\|_{L^2(\X)} \leq \frac{\|\rho\|_{L^{\infty}(\X)}\|u[q]\|_{L^2(\X)}}{\inf (q + \rho)}.
        \label{4.8}
        \end{equation}
Combining (\ref{4.7}) and (\ref{4.8}) yields (\ref{4.6}). Using
the chain rule in differentiation, we readily get (\ref{4.5}).
\end{proof}

Using regularity theory, we see that $u[q]$ belongs to
$L^{\infty}(\X)$ in the two-dimensional case. In three dimensions,
we should assume that $g \in H^{1/2}(\partial X)$ in order to
claim that $u[q] \in L^{\infty}(\X)$. Hence, $\D F[q]$ can be
extended so that its domain is $L^2(\X)$. By abuse of notation, we
denote the extended operator still by $\D F[q]$. The following key
lemma of this section establishes an estimate of the $L^2(\X)$
norm of $v(\rho)$, the solution to \eqref{eq:veq}, in terms of the
$L^2(\X)$ norm of the source $\rho u[q]$. A corollary of this
result allows us to show the invertibility of $\D F[q]$ from
$L^2(\X)$ to $L^2(\X)$.

\begin{lemma}\label{lem:vL2bdd} Assume that the origin $0$ is included in
$\X$ and define \[\diam(\X) = \sup_{x \in \partial \X} |x|.\]
Suppose that $\X$ is star-shaped and balanced with respect to the
origin so that \[x \cdot \nu_x \ge \gamma \diam(\X) \] for some
positive number $\gamma$. If
\[\|q\|_{L^\infty} \diam(\X) \le \frac{1}{4},\] and $k>2$, then
\begin{equation}
\|v(\rho)\|_{L^2} \le \eta \|\rho u[q]\|_{L^2},
 \label{eq:vL2bdd}
 \end{equation}
 where
\begin{equation}
    \eta = \sqrt{\frac{8(1+\gamma^{-1})^2 + 2d + 29}{(11-2d)}} \max \{ \diam(\X), 1\}.
 \label{eq:vL2bdd0}
 \end{equation}
\end{lemma}

\begin{proof}
Let us define the bilinear form
 \begin{equation} \label{bilinfg}
  B[v,w] = -\int_{\X} \nabla v \cdot \nabla \overline{w} dx + k^2 \int_{\X} v\overline{w} dx + ik \int_{\X} q v \overline{w} dx + ik \int_{\partial \X} v \overline{w} ds,
 \end{equation}
 and the linear form
 \begin{equation}
  G[w] = - \int_{\X} ik\rho u[q] \overline{w} dx.
 \end{equation}
Then the weak solution of \eqref{eq:veq} is characterized by $v$ satisfying
 \begin{equation}
\label{eq:weakveq}
  B[v,w] = G[w], \quad \forall w \in H^1(\X).
 \end{equation}

Using $w = v$ in (\ref{eq:weakveq}) and considering the imaginary and real parts separately, we have
 \begin{equation}
  \begin{aligned}
   \int_{\partial \X} |v|^2 ds + \int_{\X} \absp |v|^2 dx& \le \left| \int_{\X} \rho u \overline{v} dx \right|,\\
   \int_{\X} |\nabla v|^2 dx - k^2 \int_{\X} |v|^2 dx&\le k\left| \int_{\X}  \rho u \overline{v} dx \right|.
  \end{aligned}
 \label{eq:ineq_v}
 \end{equation}
It follows from these inequalities that
 \begin{equation}
  \|v\|_{L^2(\partial \X)}^2 \le \|\rho u\|_{L^2} \|v\|_{L^2},
 \label{eq:bdbdd}
 \end{equation}
 and
 \begin{equation}
  \|\nabla v\|_{L^2}^2 \le (k^2 + 1) \|v\|^2_{L^2} + \frac{k^2}{4}\|\rho u\|_{L^2}^2.
 \label{eq:Dvbdd}
 \end{equation}

To estimate $\|v\|_{L^2}$, we mimic the technique used in
\cite[Chapter 8]{JMM-thesis}.
We have
 \begin{equation*}
  \Re \big(\nabla v \cdot \nabla(x\cdot \nabla \overline{v}) \big)
  = |\nabla v|^2 + x\cdot \nabla \Big(\frac{|\nabla v|^2}{2}\Big), \quad
  \Re \big(v (x\cdot \overline{\nabla v}) \big) = x\cdot \nabla \Big(\frac{|v|^2}{2}\Big).
 \end{equation*}
 Integrating the first equation above gives
 \begin{equation*}
 \begin{aligned}
  \int_{\X} \Re \big(\nabla v \cdot \nabla(x\cdot \nabla \overline{v})\big) \, dx &= \int_{\X} |\nabla v|^2 dx + \frac{1}{2} \int_{\X} \nabla \cdot \left(x |\nabla v|^2 \right) - (\nabla \cdot x) |\nabla v|^2 dx \\
  & = \frac{1}{2} \int_{\partial \X} (\nu \cdot x) |\nabla v|^2 ds + (1-\frac{d}{2}) \|\nabla v\|_{L^2}^2.
 \end{aligned}
 \end{equation*}
The second term above is due to the fact that $\nabla \cdot x =
d$. Similarly,
\begin{equation*}
 \begin{aligned}
  k^2 \int_{\X} \Re \big( v (x\cdot \nabla \overline{v})\big) \, dx &=  \frac{k^2}{2} \int_{\X} \nabla \cdot \left(x |v|^2\right) - (\nabla \cdot x) |v|^2dx \\
  & = - \frac{d k^2}{2} \int_{\X} |v|^2 dx + \frac{k^2}{2} \int_{\partial \X} (\nu \cdot x) |v|^2 ds.
 \end{aligned}
 \end{equation*}
 Consequently, taking $w= - x \cdot \nabla v$ in (\ref{bilinfg})  we find
 \begin{equation*}
 \begin{aligned}
 - \Re B[v, x\cdot \nabla v] = \frac{d k^2}{2}
 \|v\|^2_{L^2} + \frac{1}{2} \int_{\partial \X} (x\cdot \nu) |\nabla v|^2 ds -
 \frac{k^2}{2} \int_{\partial \X} (x\cdot \nu) |v|^2 \, ds\\
 + (1 - \frac{d}{2}) || \nabla v||_{L^2}^2 + \Re \left( -ik\int_{\X} \absp v (x\cdot \nabla \overline{v})\, dx
  - ik \int_{\partial \X} v (x \cdot \nabla \overline{v}) ds \right).
 \end{aligned}
 \end{equation*}
 Equate the above expression with the real part of $ - \Re G[x \cdot \nabla v]$, {\it i.e.}, $\Re ik \int
 \rho u x\cdot \nabla \overline{v}\, dx$. We then obtain the estimate (using the fact that $x\cdot \nu \geq \gamma \diam(\X)$):
\begin{eqnarray*}
&&\hspace*{-.24in} \frac{dk^2}{2} \|v\|^2_{L^2} + \frac{\diam(\X) \gamma}{2} \|\nabla v\|^2_{L^2(\partial \X)} \le \frac{k^2 \diam(\X)}{2}
\|v\|^2_{L^2(\partial \X)} + (\frac{d}{2}-1) \|\nabla v\|_{L^2}^2 \\
&& \h + k\ \diam(\X) \left( \|\absp\|_{L^\infty} \|v\|_{L^2} \|\nabla v\|_{L^2} + \|v\|_{L^2(\partial \X)} \|\nabla v\|_{L^2(\partial \X)} + \|\rho u\|_{L^2} \|\nabla v\|_{L^2} \right).
\end{eqnarray*}

On the other hand, it follows from Young's inequality that
 \begin{equation*}
  \|v\|_{L^2(\partial \X)}\|\nabla v\|_{L^2(\partial \X)} \le \epsilon \|\nabla v\|_{L^2(\partial \X)}^2 + \frac{1}{4\epsilon} \|v\|_{L^2(\partial \X)}^2,
 \end{equation*}
for all $\epsilon > 0$. We choose $\epsilon$ such that $k\epsilon  = \gamma/2$ to get
 \begin{eqnarray*}
 && \hspace*{-.24in}  \frac{k^2 \diam(\X)}{2} \|v\|_{L^2(\partial \X)}^2 + k\ \diam(\X)\|v\|_{L^2(\partial \X)}\|\nabla v\|_{L^2(\partial \X)} \\
 && \h \le \frac{\gamma \diam(\X)}{2} \|\nabla v\|_{L^2(\partial \X)}^2 + \frac{k^2 \diam(\X)}{2} \frac{\gamma+1}{\gamma} \|v\|_{L^2(\partial \X)}^2.
 \end{eqnarray*}
 Recall \eqref{eq:bdbdd}. The left-hand side of the inequality above can be further bounded by
 \begin{equation}
  \frac{\gamma \diam(\X)}{2} \|\nabla v\|_{L^2(\partial \X)}^2 + \frac{k^2 \diam(\X)}{2} \frac{\gamma+1}{\gamma} (\epsilon_1 \|v\|_{L^2}^2 + \frac{1}{4\epsilon_1} \|\rho u\|_{L^2}^2).
 \label{eq:bterm}
 \end{equation}

Applying Young's inequality to the term $\|\rho u\|_{L^2} \|\nabla
v\|_{L^2}$ with $\epsilon k \diam(\X) = 1/8$ yields
\begin{equation}
  k\diam(\X)\|\rho u\|_{L^2} \|\nabla v\|_{L^2} \le \frac{1}{8}\|\nabla v\|_{L^2}^2 + 2 k^2 \diam^2(\X) \|\rho u\|_{L^2}^2.
 \label{eq:fterm}
 \end{equation}

Applying the same technique to the term $\|v\|_{L^2} \|\nabla v\|_{L^2}$ shows
 \begin{equation}
  k\diam(\X)\|\absp\|_{L^\infty} \|v\|_{L^2} \|\nabla v\|_{L^2} \le \frac{1}{8}\|\nabla v\|_{L^2}^2 + 2k^2 \diam^2(\X) \|\absp\|_{L^\infty}^2 \|v\|_{L^2}^2.
 \label{eq:vterm}
 \end{equation}

Finally, recalling estimate \eqref{eq:Dvbdd} and combining the
above inequalities, we have
 \begin{equation}
 \begin{aligned}
  \frac{d}{2}\|v\|_{L^2}^2 \le  \left(\frac{\diam(\X)}{2}
  \frac{\gamma+1}{\gamma} \epsilon_1 + { (\frac{d}{2} - \frac{3}{4}) (1+k^{-2})}
  + 2 \|\absp\|^2_{L^\infty} \diam^2(\X) \right) \|v\|_{L^2}^2 \\
  + \left(\frac{\diam(\X)}{8\epsilon_1} \frac{\gamma+1}{\gamma}
  + 2\diam^2(\X) + \frac{1}{4}(\frac{d}{2} - \frac{3}{4})\right) \|\rho u\|^2_{L^2}.
  \end{aligned}
 \end{equation}

Suppose that the wave number $k$ is larger than $2$ and the
product $\|\absp \|_{L^\infty} \diam(\X)$ is smaller
 than $1/4$. Then, if $4\epsilon_1$ is chosen to be $(\diam(\X)(\gamma+1)/\gamma)^{-1}$, the coefficient in front
 of $\|v\|_{L^2}^2$ on the right is less than $5d/8 - 11/16$. Then $\|v\|_{L^2}$ term on the left dominates and we have
\begin{equation*}
(\frac{11}{16} - \frac{d}{8})\|v\|_{L^2}^2 \le  \left(
 \frac{(\gamma+1)^2}{2 \gamma^2} \diam^2(\X) + 2 \diam^2(\X) + \frac{1}{4}
(\frac{d}{2} - \frac{3}{4}) \right) \|\rho u\|_{L^2}^2.
\end{equation*}
Estimate \eqref{eq:vL2bdd} follows from this immediately.
\end{proof}

\begin{lemma}\label{lem:DFinv}
Let $\eta$ denote the constant \eqref{eq:vL2bdd0}.
Suppose that the absorption coefficient $q$ is such that
\begin{equation}
\eta \|q\|_{L^\infty (\X)} < \frac{1}{4}. \label{4.23}
\end{equation}
Suppose also that $|u[q]|$ is bounded from below by a positive
number. Then the map $\D F[q]$, as an operator from $L^2(\X)$ to
$L^2(\X)$, is invertible. Moreover,
\begin{equation}
    \|D F[q]^{-1}\|_{\mathcal{L}(L^2(\X))} \leq \frac{1}{\inf |u[q]|^2\sqrt{1 - 4\eta \|q\|_{L^{\infty}(\X)}}}.
\label{ine:DFinve}
\end{equation}
\end{lemma}
\begin{proof}
Define
\[
    T[q](\rho) = |u[q]|^{-2}\D F[q](\rho) - \rho . 
\]
It is not hard to see that $T$ is compact since it can be decomposed as
\[
\hspace*{-.24in}\begin{array}{rclclclccccccccc}
         T: L^2(\X) &\rightarrow& H^1(X) &\hookrightarrow& L^2(\X) &\rightarrow& L^2(\X)\\
        \rho &\mapsto& v(\rho) &\mapsto& v(\rho) &\mapsto& 2q|u[q]|^{-2}\Re (u[q] \overline
        v(\rho)).
    \end{array}\\
\]
The continuity of maps in the diagram above can be deduced from Proposition \ref{proposition 2.6} and the choice of $g$ such that $|u[q]| > 0$ in $\overline \X$.

On the other hand, a straightforward calculation shows that
\begin{equation}
    \|\D F[q](\rho) \|^2_{L^2(\X)} \geq \inf|u[q]|^4\|\rho \|_{L^2(\X)}^2\left[1 -
    4\eta\|q\|_{L^{\infty}(\X)}\right].
\label{4.24}
\end{equation}
In fact,
\begin{eqnarray*}
\|\D F[q](\rho) \|^2_{L^2(\X)} &=& \io \left[\rho^2|u[q]|^4 + 4q^2\Re^2 (u[q]\overline v(\rho)) + 4q\rho|u[q]|^2\Re (u[q]\overline v(\rho))\right]dx\\
&\geq& \inf|u[q]|^2 \io \left[\rho^2|u[q]|^2 + \frac{4q^2\Re^2 (u[q]\overline v(\rho))}{|u[q]|^2} + 4q\Re (\rho u[q]\overline v(\rho))\right]dx\\
&\geq& \inf|u[q]|^2 \io \left[\rho^2|u[q]|^2 - 4\|q\|_{L^{\infty}(\X)}|\rho u[q]\overline v(\rho)|\right]dx\\
&\geq& \inf|u[q]|^2\left[\|\rho |u[q]|\|_{L^2(\X)}^2 - 4\|q\|_{L^{\infty}(\X)}\|\rho u[q]\|_{L^2(\X)}\|v(\rho)\|_{L^2(\X)}\right]\\
&\geq& \inf|u[q]|^2\|\rho u[q]\|_{L^2(\X)}^2\left[1 -
4\eta\|q\|_{L^{\infty}(\X)}\right].
\end{eqnarray*}
Since $\eta \|q\|_{L^{\infty}(\X)} < 1/4$, we find (\ref{4.24}).
It follows that the kernel of $\D F[q]$ is $\{0\}$. Hence, by the
Fredholm theory, $\D F[q]$ is invertible. Moreover, (\ref{4.24})
also implies (\ref{ine:DFinve}).
\end{proof}
\begin{remark} Recall the definition of $\eta$ in (\ref{eq:vL2bdd0}). When $\X$ is a ball, $\eta$ is roughly three to four times the radius of $\X$ in dimensions three or two.
Condition \eqref{4.23} hence requires that $\|q\|_{L^\infty}
\diam(\X)$, which can be interpreted as the typical absorption
rate as signals propagate to the boundary, should be sufficiently
small.
\end{remark}

\section{Measurement noise and resolution enhancement}
\label{sec:noise}

In this section, we consider additive noise in the data set
$\mathcal{E}$ given in \eqref{eq:dataset}.

\subsection{Noise model}

As described in Proposition \ref{proposition 2.5}, the data
$\mathcal{E}$ are acquired by measuring several sets of absorbed
radiations: $q|u_1+u_j|^2$, $q|iu_1 + u_j|^2$, $q|u_1|^2$, and
$q|u_j|^2$ for $j = 2, \ldots, d+1$. In practice, the measurements
of these absorbed energies are corrupted by additive noises. We
model a typical energy measurement by
\begin{equation}
\Em(x) = E(x) + \sigma W_\delta(x)  . \label{eq:Emnoise}
\end{equation}
Here and in the sequel, the superscript ``m" indicates measured
quantity, and $E$ itself is the pure quantity without noise.
$W_\delta$ is a stationary random field with mean zero and covariance function of the form
\begin{equation}
\E  \big[W_\delta(x) W_\delta(y) \big] = \E \big[ W_\delta(0)W_\delta(x-y) \big]= R\Big(\frac{x-y}{\delta}\Big),
 \label{eq:Rdef}
\end{equation}
where $R$ is an integrable function normalized so that $R(0)=1$.
In this additive noise model,
$\sigma^2$ is the variance of the noise, $\delta$ is the correlation length which is related to the distance between measurement points.

The random process $W_\delta$ is assumed to be bounded almost surely
by a constant 
 independent of $\delta$.
This constant is assumed to be smaller than $E_{\rm min}$ which is a lower bound for the real energy.
This technical hypothesis ensures that $\Em$ is bounded from below by a positive constant
for any $\sigma \leq 1$ and for any $\delta$.

In the forthcoming analysis, both the noise variance $\sigma$ and the noise correlation length $\delta$ will be supposed to be small.
We assume that the measured data
$\mathcal{E}^\mathrm{m} = (\Em_j )_{j=1}^{d+1}$ are given by
\begin{equation}
\begin{aligned}
\Em_1(x ) &= E_1(x) + \sigma W_{\delta 1}(x ), \\
\Em_j(x ) &= E_j(x) + \sigma U_{\delta j}(x) +
i\sigma V_{\delta j}(x), \quad \quad j = 2, \ldots, d+1.
\end{aligned}
\label{eq:Enoise}
\end{equation}
According to the procedure of measuring $E_j$, the random fields $U_{\delta j}$ and $V_{\delta j} $  are given by
$(W_{\delta 1j} - W_{\delta 1} - W_{\delta j})/2$ and $(W_{\delta 1j'} - W_{\delta 1} - W_{\delta j})/2$
respectively, where $W_{\delta j}$, $W_{\delta 1j}$ and $W_{\delta 1j'}$ correspond to the
additive noises of the energy measurements $q|u_j|^2$, $q|u_1 +
u_j|^2$ and $q|iu_1 + u_j|^2$, respectively. It is natural to
assume that $W_{\delta 1}$, $W_{\delta j}$, $W_{\delta 1j}$ and $W_{\delta 1j'}$ are mutually
independent and have the same statistical distribution as $W_\delta$ in
\eqref{eq:Emnoise}. As a consequence, $U_{\delta j}$, $V_{\delta j}$ and $W_{\delta 1}$ are
correlated.

\subsection{Initial guess with smoothed data}
\label{sec:igsd}

We smooth the data $\mathcal{E}$ by using the convolution kernel
\begin{equation}
\varphi_\delta (x) : = \frac{1}{\delta^{dp}}
\varphi(\frac{x}{\delta^p}), \label{eq:vphidelta}
\end{equation}
where $p \in (0,\frac{d}{d+6})$ and $\varphi$ is in the Schwartz space of smooth nonnegative functions that
decay rapidly at infinity and that satisfy $\int_{\R^d} \varphi(x) dx=1$.
 The condition $p < d/(d+6)$ will be clear later.  The
following lemma will be useful.

\begin{lemma}
\label{lem:Wcphi}%
Let $|\gamma|$ denote the sum of all components of the multi-index
$\gamma$ and $\partial^\gamma \varphi$ (resp. $\partial^\gamma \varphi_{\delta}$)
denotes the usual $\gamma-$partial derivative of $\varphi$ (resp. $\varphi_{\delta}$).
For any $\delta$ we have
\begin{equation}
\E \lvert W_{1\delta} * \partial^\gamma \varphi_\delta \rvert^2
\leq   \delta^{d-(d+2|\gamma|)p} \| R \|_{L^1(\R^d)}
\| \partial^{\gamma} \varphi \|_{L^2(\R^d)}^2.
\label{eq:Wcphi0}
\end{equation}
More precisely, for $\delta \ll 1$, we have
\begin{equation}
\E \lvert W_{1\delta} * \partial^\gamma \varphi_\delta \rvert^2
= \delta^{d-(d+2|\gamma|)p} \int_{\R^d} R(y) dy
\int_{\R^d} |\partial^{\gamma} \varphi (y')|^2 dy'  +
o\big(  \delta^{d-(d+2|\gamma|)p} \big).
\label{eq:Wcphi}
\end{equation}
\end{lemma}
\begin{proof} The variance (\ref{eq:Wcphi0}) can be written as
\begin{equation*}
\begin{aligned}
\E \lvert W_{1\delta} * \partial^\gamma \varphi_\delta \rvert^2=
& \E \frac{1}{\delta^{2p|\gamma|+2dp}} \int_{\R^d} \int_{\R^d} W_{1\delta}(x - y) W_{1\delta}(x - y') (\partial^\gamma \varphi)(\frac{y}{\delta^p}) (\partial^\gamma \varphi)(\frac{y'}{\delta^p}) dy dy'\\
= & \frac{1}{\delta^{2p|\gamma|+2dp}} \int_{\R^d}
\int_{\R^d} R(\frac{y-y'}{\delta}) (\partial^\gamma
\varphi)(\frac{y}{\delta^p}) (\partial^\gamma
\varphi)(\frac{y'}{\delta^p}) dy dy'.
\end{aligned}
\end{equation*}
We apply the change of variable $(y-y')/\delta \to y'$ and  $y/\delta^p \to
y$, and take advantage of the resulting Jacobian. We verify that
the variance can be written as
\begin{equation*}
\E \lvert W_{1\delta} * \partial^\gamma \varphi_\delta \rvert^2=
  \delta^{d + dp -2p|\gamma| - 2dp} \int_{\R^d} R(y') \int_{\R^d}
(\partial^\gamma \varphi)(y) (\partial^\gamma \varphi)(y -
\delta^{1-p} y') dy dy'.
\end{equation*}
Using Cauchy-Schwarz inequality and the fact that $\partial^\gamma \varphi \in L^2$ and $R \in L^1$,
we obtain \eqref{eq:Wcphi0}.
Since $\partial^\gamma \varphi \in L^2$, $p<1$, and $R$ is integrable,
 \eqref{eq:Wcphi} is also easily verified by the dominated convergence theorem.
\end{proof}
\begin{remark}
The above calculation works also for $U_{j\delta}$ and
$V_{j\delta}$.
\end{remark}

We smooth the data by evaluating the convolution with the kernel $\varphi_\delta$:
\begin{equation}
 \Es_j = \Em_j * \varphi_\delta , \quad j = 1, \ldots, d+1,
\label{eq:Esdef}
\end{equation}
which gives
\begin{eqnarray}
E_{1}^{\rm s} &=& E_1 * \varphi_{\delta} + \sigma W_{1\delta} * \varphi_{\delta}  ,\\
\Es_j &=& E_j * \varphi_\delta + \sigma
U_{j\delta} * \varphi_\delta + i \sigma V_{j\delta} *
\varphi_\delta, \quad j = 2, \ldots, d+1.
\end{eqnarray}
Here and below, the superscript ``s" indicates smoothed
quantities. The parameter $\delta^p$ can be interpreted as the
size of the averaging window. To simplify the notation,
$E_{j\delta}$ will be used as the short-hand notation for the
smoothed unperturbed data $E_j * \varphi_\delta$ in the sequel.

\begin{proposition}
If we substitute the smoothed measured data $(\Es_j)_{j=1}^{d+1}$
into the reconstruction formula \eqref{2.12}:
\begin{equation}
\label{2.12s}
 q^{\rm s}(x)  =   \frac{-\Re ( a^{\rm s} ) \Im (a^{\rm s} ) + \mbox{div} \,  \Im  (a^{\rm s}) }{2k}  ,
\end{equation}
with   $a^{\rm s}= {(A^{\rm s})}^{-1}[ (\nabla^T {(A^{\rm
s})}^T)^T ]$, $A^{\rm s}= (\partial_l \alpha^{\rm s}_{j+1}
)_{j,l=1,\ldots,d}$, and $\alpha^{\rm s}_j=E_j^{\rm s}/E_1^{\rm
s}$,
 then the estimate $q^{\rm s}$ satisfies:
\begin{equation}
\sup_{x \in X} \E\big[ | q^{\rm s}(x)  - q_\delta(x)|^2 \big] \leq
C\sigma^2 \delta^{d-(d+6)p} ,
\end{equation}
where
$$
q_\delta(x)= \frac{-\Re ( a^\delta ) \Im (a^\delta ) +  \mbox{div} \, \Im (a^\delta) }{2k}
$$
is obtained by substituting the smoothed unperturbed data
$(E_{j\delta})_{j=1}^{d+1}$ into the reconstruction formula
\eqref{2.12}.
\end{proposition}

\begin{proof}
We substitute the smoothed data $(\Es_j)_{j=1}^{d+1}$ into the
reconstruction formula \eqref{2.12}. Recall the definitions of $A$
and $\alpha_j$ in \eqref{2.888} and \eqref{2.2}. Then,
\begin{equation}
\label{eq:alphas}
\alpha^\mathrm{s}_j = \frac{E_{j\delta} + \sigma
U_{j\delta}*\varphi_\delta + i \sigma V_{j\delta}*
\varphi_\delta}{E_{1\delta} + \sigma W_{1\delta}*\varphi_\delta}.
\end{equation}
When $\sigma \ll 1$, we can linearize this term and find that
\begin{equation}
\alpha^\mathrm{s}_j = \frac{E_{j\delta}}{E_{1\delta}} - \sigma
\frac{W_{1\delta}*\varphi_\delta}{E_{1\delta}}
\frac{E_{j\delta}}{E_{1\delta}} + \sigma
\frac{U_{j\delta}*\varphi_\delta}{E_{1\delta}} + i\sigma
\frac{V_{j\delta}*\varphi_\delta}{E_{1\delta}} + O(\sigma^2).
\label{eq:alphaexp}
\end{equation}
The coefficients of the matrix $\As$ are defined by $\As_{jl}  =  \partial_l \alphas_{j+1}$
and they can be expanded from (\ref{eq:alphas}) as
\begin{equation}
\As_{jl}  =   A^\delta_{jl} + \sigma A^{\delta(1)}_{jl} + o(\sigma
\delta^{d/2-(d+2)p/2}), \h   1 \leq j, l \leq d ,
\label{eq:Aexp}
\end{equation}
where
$$
A^{\delta}_{jl}  =
 \partial_l \frac{E_{j+1\delta}}{E_{1\delta}} ,\quad \quad
 A^{\delta(1)}_{jl}  =  -   \frac{W_{1\delta}* \partial_l \varphi_\delta}{E_{1\delta}} \frac{E_{j+1\delta}}{E_{1\delta}} +   \frac{U_{j+1\delta}* \partial_l \varphi_\delta}{E_{1\delta}} + i \frac{V_{j+1\delta}* \partial_l \varphi_\delta}{E_{1\delta}}
 .
$$
The leading-order error terms $ \sigma A^{\delta(1)}_{jl} $ have zero means and
their variances are of order $O(\sigma^2 \delta^{d - (d+2)p})$
according to Lemma \ref{lem:Wcphi}, provided that the functions
$E_j$'s are sufficiently smooth with bounded derivatives.
The following error terms like $  W_{1\delta}*  \varphi_\delta \partial_l  \big( \frac{E_{j+1\delta}}{E_{1\delta}^2}\big) $
are smaller since their square means are of order $O(\sigma^2 \delta^{d - dp})$.

Since $A^\delta$ is a smoothed version of $A$, which was defined in
\eqref{2.888} and whose determinant can be bounded from below by a
large constant (see Proposition \ref{proposition 2.3}), the
inverse of $A^\delta$ is well defined. Linearizing
$(A^\mathrm{s})^{-1}$, we have
\begin{equation*}
(A^\mathrm{s})^{-1} = (A^\delta)^{-1} + \sigma (A^\delta)^{-1}
A^{\delta(1)} (A^\delta)^{-1} + o(\sigma \delta^{d/2 -(d+2)p/2}).
\end{equation*}
Similarly, the vector $ (\nabla^T {A^\mathrm{s}}^T)^T$ can be decomposed as
\begin{equation*}
\begin{aligned}
(\nabla^T {A^\mathrm{s}}^T)_j = & (\nabla^T {A}^T)_j +
\sigma \left( - \frac{W_{1\delta}* \Delta \varphi_\delta}{E_{1\delta}} \frac{E_{j+1\delta}}{E_{1\delta}} +  \frac{U_{j+1\delta}*\Delta \varphi_\delta}{E_{1\delta}} + i \frac{V_{j+1\delta}*\Delta \varphi_\delta}{E_{1\delta}}\right) \\
& \quad + o(\sigma \delta^{d/2-(d+4)p/2}).
\end{aligned}
\end{equation*}
Finally, we have for the vector $a^{\rm s} = {(A^\mathrm{s})}^{-1}
(\nabla^T {A^\mathrm{s}}^T)^T$:
\begin{equation*}
\begin{aligned}
a^{\rm s}_j = &  a^{\delta} _j +  \sigma  \sum_{l=1}^d (A^\delta)^{-1}_{jl} \left(
- \frac{W_{1\delta}* \Delta \varphi_\delta}{E_{1\delta}} \frac{E_{l+1\delta}}{E_{1\delta}} + \frac{U_{l+1\delta}*\Delta \varphi_\delta}{E_{1\delta}} +
  i \frac{V_{l+1\delta}*\Delta
\varphi_\delta}{E_{1\delta}}\right) \\
&+ o(\sigma
\delta^{d/2-(d+4)p/2}),
\end{aligned}
\end{equation*}
and
\begin{equation*}
\begin{aligned}
\mbox{div} \, a^{\rm s} =
  &  \mbox{div} \, a^\delta +  \sigma \sum_{j,l=1}^d (A^\delta)^{-1}_{jl} \left(
  - \frac{W_{1\delta}* (\partial_j \Delta \varphi)_\delta}{E_{1\delta}} \frac{E_{l+1\delta}}{E_{1\delta}} + \frac{U_{l+1\delta}*(\partial_j\Delta \varphi)_\delta}{E_{1\delta}} + i \frac{V_{l+1\delta}*(\partial_j\Delta
\varphi)_\delta}{E_{1\delta}}\right) \\
&+ o(\sigma
\delta^{d/2-(d+6)p/2}).
\end{aligned}
\end{equation*}
The vector  $a^{\delta} = {(A^\delta)}^{-1} (\nabla^T
{A^\delta}^T)^T$ is obtained by applying formulas (\ref{2.2}) and
(\ref{2.888}) to the smoothed unperturbed data
$(E_{j\delta})_{j=1}^{d+1}$. The leading-order error terms have
zero means and their variances are of order $ O(\sigma^2 \delta^{d
- (d+6)p})$ according to Lemma \ref{lem:Wcphi}. Our choice $p <
\frac{d}{d+6}$ guarantees that the noisy data are smoothed enough
so that the terms above have variance of order smaller than
$\sigma^2$. To summarize, if we apply (\ref{2.12}) to the smoothed
data $(\Es_j)_{j=1}^{d+1}$, then we get
\begin{equation}
\label{eq:iguess}
\begin{aligned}
 q^{\rm s}(x)  = & q_\delta(x)
 - \frac{\sigma}{2k} \bigg\{ \sum_{j,l=1}^d  \Im (A^\delta)^{-1}_{jl} \left( - \frac{W_{1\delta}*  \partial_j \Delta \varphi_\delta}{E_{1\delta}} \frac{E_{l+1\delta}}{E_{1\delta}} +\frac{U_{l+1\delta}* \partial_j\Delta \varphi_\delta}{E_{1\delta}} \right)  \\
& +  \Re (A^\delta)^{-1}_{jl}
\frac{V_{l+1\delta}* \partial_j\Delta \varphi_\delta}{E_{1\delta}}
\bigg\} +  o(\sigma \delta^{d/2-(d+6)p/2}),
\end{aligned}
\end{equation}
from which we deduce the desired result.
\end{proof}

The terms $q_\delta$ can be shown to be close to the real absorption parameter
$q_o$ uniformly in $x$ (we show this in Theorem \ref{thm:iguessbdd}). However, it is impossible to separate
$q_\delta$ from the noise, that is the other terms in (\ref{eq:iguess}).
Nevertheless, the estimate $q^{\rm s}$ is a good initial guess in the mean square sense
as shown by the following theorem.

\begin{theorem}
\label{thm:iguessbdd}%
Suppose that the pure data $(E_j)_{j=1}^{d+1}$ belong to $\mathcal{C}^{3,\eps}$ for
some positive real number $\eps$. Then, we have
\begin{equation}
\label{eq:lemigu1} \|q_\delta - q_o\|_{L^\infty(\X)} \le C
\delta^{\eps p}.
\end{equation}
As a result, estimate (\ref{2.12s})  obtained from the smoothed
data satisfies
\begin{equation} \label{eq:lemigu2}
\sup_{x \in X} \E \big[ |q^{\rm s}(x)- q_o(x) |^2  \big] \le C \big( \delta^{2 \eps p} + \sigma^2
\delta^{d-(d+6)p} \big).
\end{equation}
\end{theorem}

\begin{proof} Under the conditions of the theorem, the inequality $|\partial^\gamma E_j(x-y) - \partial^\gamma E_j(x)|
\le C|y|^{\eps}$ holds for some constant $C$ and for any
multi-index $\gamma$ with $|\gamma| \le 3$. As a result, we have
the following estimate as an analog of Lemma \ref{lem:Wcphi}:
\begin{equation}
\begin{aligned}
|\partial^\gamma E_{j\delta}(x) - \partial^\gamma E_j(x)| & =  \left|\frac{1}{\delta^{dp}} \int_{\R^d} (\partial^\gamma E_j(x-y) - \partial^\gamma E_j(x)) \varphi(\frac{y}{\delta^p}) dy\right| \\
&\le C \frac{1}{\delta^{dp}} \int_{\R^d} |y|^{\eps}
|\varphi(\frac{y}{\delta^p})| dy = C \delta^{\eps p} \int_{\R^d}
|y|^{\eps} |\varphi(y)| dy \le C\delta^{\eps p}.
\end{aligned}
\label{eq:Ecphi}
\end{equation}
Then the estimate of $q_\delta$ follows because the reconstruction
formula in \eqref{2.12} depends continuously on the data and their
derivatives. For the second estimate, we apply the triangle
inequality and use the control of the stochastic terms in the
linearization procedure.
\end{proof}

\begin{remark} Estimate (\ref{eq:lemigu2})  is a bit over pessimistic. Indeed, it does not imply that $q^{\rm s}$ is positive, which is a physical constraint for the absorption parameter. We will exploit this remark in the next section.
\end{remark}
\subsection{The optimization step and resolution enhancement}
\label{sec:osre}

Now we refine the above initial guess $q^{\rm s}$ by an optimal control
approach. We seek for the least square estimate of the discrepancy
functional
\begin{equation}
J[q] = \int_\X \big|F[q](x) - E^\mathrm{s}_1(x)\big|^2 dx. \label{eq:Jdef}
\end{equation}
Here, $E^\mathrm{s}_1$ is the smoothed data $(E_1 + \sigma
W_{1\delta}) * \varphi_\delta$ and $F[q] = q|u_1[q]|^2$ is the
absorbed heat energy with boundary condition $g_1$.

In Theorem \ref{thm:iguessbdd}, the initial guess $q^{\rm s}$ is
shown to be close to the true absorption coefficient. This allows
us to approximate the integrand in the definition of $J$ by its
linearization around $q^{\rm s}$; that is,
\begin{equation}
J[q] \approx \int_\X |\D F[q^{\rm s}] (q-q^{\rm s}) - b^{\rm s}|^2 dx,
\label{eq:Jlinear}
\end{equation}
where $b^{\rm s} = E^\mathrm{s}_1 - F[q^{\rm s}]$ is the residue. In the case
when $\D F[q^{\rm s}]$ is invertible from $L^2$ to $L^2$, the least
square solution of the approximate discrepancy functional is given
by
\begin{equation}
q_* = q^{\rm s} + (\D F[q^{\rm s}])^{-1} b^{\rm s}.
 \label{eq:qstardef}
\end{equation}
The following result shows that $q_*$ is a refinement of $q^{\rm
s}$ in the mean square sense (compared to Theorem
\ref{thm:iguessbdd}).

\begin{theorem}
\label{thm:refine}%
Recall that $q_o$ denotes the true absorption coefficient and
assume that the condition in Theorem \ref{thm:iguessbdd} holds. We
have
\begin{equation}
\E \big[ \|q_* - q_o\|_{L^2(\X)}^2 \big] = o(\delta^{2 \eps p} +
\sigma^2 \delta^{d-(d+6)p}) + O(\delta^{2p} + \sigma^2
\delta^{d(1-p)}). \label{eq:qstarbdd}
\end{equation}
\end{theorem}

\begin{proof} From the definition of $b^{\rm s}$ and $E^\mathrm{s}_1$, the residue can be expanded as
\begin{equation*}
b^{\rm s}= E_1 - F[q^{\rm s}] + (E_{1\delta} - E_1) + \sigma W_{1\delta} *
\varphi_\delta.
\end{equation*}
Since $E_1 = F[q_o]$, the difference $F[q_o] - F[q^{\rm s}]$ can be
linearized as $\D F[q^{\rm s}] (q_o - q^{\rm s}) + o(q_o - q^{\rm s})$. This,
together with \eqref{eq:qstardef}, implies
\begin{equation}
q_* - q_o = (\D F[q^{\rm s}])^{-1}  \{ \sigma W_{1\delta} *
\varphi_\delta + (E_{1\delta} - E_1) + o(q_o - q^{\rm s}) \}.
\label{eq:qerror}
\end{equation}
Lemma \ref{lem:Wcphi} shows that $\sigma W_{1\delta} *
\varphi_\delta$ has mean square of order $\sigma^2
\delta^{d(1-p)}$; the calculation in \eqref{eq:Ecphi} shows that
$E_{1\delta} - E_1$ can be bounded uniformly by $C\delta^p$; the
term $q_o - q^{\rm s}$ is also controlled in \eqref{eq:lemigu2}.
Consequently, since $\D F[q^{\rm s}]$ has bounded inverse (see Lemma
\ref{lem:DFinv}), the desired estimate holds.
\end{proof}

\begin{remark}
    Assume that $q_o$ is bounded from below and above by two known positive numbers $q_{\min}$ and $q_{\max}$. Let
    \[
         \hat q_{*} = \min\left\{\max\{q_*, q_{\min}\}, q_{\max}\right\} \in [q_{\min}, q_{\max}].
    \]
    We can  see that
    \[
        \|\hat q_* - q_o\|_{L^2(X)} \leq \|q_* - q_o\|_{L^2(\X)}.
    \] We note that there is no guarantee that $q_*$ is positive, but the modified version $\hat q_*$ is. In addition to this advantage, the estimate above shows that $\hat q_*$ is a better approximation of $q_o$ in comparison with $q_*$. Further, the range of $\hat q_*$ allows us to make iterations for further corrections.
\end{remark}

\begin{remark} Finally, we note that the above result also shows that the optimization step enhances the resolution.
In fact, from (\ref{eq:qstardef}) it follows that $q_*$ contains
higher oscillations than $q^{\rm s}$ and therefore, yields a more
resolved approximation of $q_o$.
\end{remark}

\section{Conclusion}

In this paper we have derived an exact reconstruction formula for
the absorption coefficient from thermo-acoustic data associated
with a proper set of measurements. Using a noise model for the
data, we have regularized this formula in order to obtain a good
initial guess. We have also performed a refinement of the initial
guess using an optimal control approach and shown that this
procedure reduces the occurring errors and yields a resolution
enhancement. A challenging problem is to estimate analytically the
resolution. It would be also very interesting to study the
reconstruction problem in the case of incomplete measurements,
where the thermal energy is known only on an open subset of the
domain. The numerical implementation of our approach in this paper
is the subject of forthcoming work, which will be published
elsewhere.


\end{document}